\documentclass[12pt]{amsart}
\title[Wandering domains of tori]{No round wandering domains for ${\bf C^1}$-diffeomorphisms of tori}

\author{Sergei Merenkov}

\address{Department of Mathematics\\City College of New York and CUNY Graduate Center,  New York, NY 10031, USA.
}
\email{smerenkov@ccny.cuny.edu}

\usepackage{amsmath, amssymb, amsthm,latexsym}
\usepackage{graphicx}

\newcommand\N{{\mathbb N}}

\newcommand\Z{{\mathbb Z}}
\newcommand\R{{\mathbb R}}
\newcommand\T{{\mathbb T}}

\newcommand\Sph{{\mathbb S}}

\newcommand\dee{\partial}

\newcommand\id{\operatorname{id}}

\renewcommand\:{\colon}

\newcommand\no{\noindent} 

\newtheorem{theorem}{Theorem}[section]

\newtheorem{conjecture}[theorem]{Conjecture}

\newtheorem{lemma}[theorem]{Lemma}

\theoremstyle{definition}

\begin{document}


\abstract{We prove that if $n\geq 2$, then there is no $C^1$-dif\-feo\-mor\-phism $f$ of $n$-torus, such that $f$ is semi-conjugate to a minimal translation and its wandering domains are geometric balls. This improves a recent result of A. Navas, who proved it assuming $C^{n+1}$ regularity of $f$.} 
\endabstract

\maketitle

\section{Introduction}\label{s:Intro}
\no 
In a recent paper, A.~Navas~\cite{Na17} proved non-existence  of a $C^{n+1}$-diffeomorphism $f$ of $n$-torus, $n\geq 2$, such that $f$ is semi-conjugate to a minimal translation and has round wandering domains, i.e., the wandering domains are non-trivial geometric balls. He also pointed out that it is unknown whether $C^{n+1}$ regularity is needed under the assumption that the wandering domains are round. 

In the present note we show that indeed one can lower regularity to $C^1$ for all $n\geq 2$. Namely, the main result of this note is the following theorem; please see Section~\ref{S:Prelim} for the definition of Denjoy-type homeomorphisms.

\begin{theorem}\label{T:Main}
If $n\geq 2$, there is no Denjoy-type $C^1$-dif\-feo\-mor\-phism $f$ of $\T^n$ all of whose wandering domains are round.
\end{theorem}

A classical result of A.~Denjoy~\cite{De32} states  that the action of an orientation preserving $C^2$-diffeomorphism of the circle with irrational rotation number is minimal.
The question of whether a homeomorphism of 2-torus that is semi-conjugate (but not conjugate) to a mi\-ni\-mal translation can have uniform conformal geometry along orbits of preimages of points under the semi-conjugation originates from a paper by A.~Norton and D.~Sullivan~\cite{NS96}. They proved, in particular, that no $C^3$-diffeomorphism $f$ of 2-torus exists such that the wandering domains of $f$ are non-degenerate round discs. 
Our Theorem~\ref{T:Main} gives a generalization  of this result to higher dimensions and under a lower regularity assumption.

Norton--Sullivan's proofs rely on Ahlfors--Bers and Sullivan's integrability results for quasiconformal maps. 
In contrast, A.~Navas uses Brouwer's fixed point theorem to arrive at a contradiction. 
The arguments of the present paper are different from both of these and use techniques developed in joint work of the author with M.~Bonk and B.~Kleiner~\cite{BKM09}.  The main tools here are reflection groups, also known as Schottky groups. We use such groups to redefine a lift of a hypothetical map to be equivariant. Equivarience implies conformality, which leads to a contradiction. 

Finally, we would like to mention works of P.~D.~McSwiggen~\cite{McS93, McS95}, where he constructs $C^{n+1-\epsilon}$-diffeomorphisms of $n$-tori that are semi-conjugate to minimal translations and have wandering domains that are topological discs. These wandering domains have non-uniform geometry. McSwiggen's results are higher-dimensional generalizations of classical results by P.~Bohl~\cite{Bo16}, A.~Denjoy~\cite{De32}, and M.-R.~Herman~\cite{He79} who established them for circle diffeomorphisms. For a construction of $C^{3-\epsilon}$-diffeomorphisms $f$ of 2-torus $\T^2$ similar to that of McSwiggen, but so that such maps $f$ possess additional properties, please consult~\cite{PS13}. In this case $f$ is homotopic and semi-conjugate (but not conjugate) to a minimal translation, and for a semi-conjugacy $h$ one has that for each $p\in\T^2$, the set $\{h^{-1}(p)\}$ is either a point or an arc. Moreover, there
are uncountably many points $p$ such that $\{h^{-1}(p)\}$ is a nontrivial arc.

\medskip\noindent 
\textbf{Acknowledgments.} The author thanks Andr\'es Navas for his comments on an earlier version of this paper.

\section{Preliminaries}\label{S:Prelim}
\no
\subsection{Denjoy-type homeomorphisms}\label{SS:Denjoy}
By $n$-\emph{torus} we mean the quotient $\T^n=\R^n/\Z^n$ of the $n$-dimensional Euclidean space $\R^n$ by the integer lattice $\Z^n$. In other words, $\T^n$ consists of equivalence classes of elements $x\in\R^n$, where $x\sim y$ if and only if $x-y=m\in\Z^n$.  We denote the quotient map by $\psi$. A \emph{minimal translation} $R$ of $\T^n$ is a homeomorphism of $\T^n$ such that $R$ lifts under $\psi$ to a translation $\tilde R$ of $\R^n$, and every point $p\in \T^n$ has a dense orbit under the action of the map $R$.  

Following~\cite{NS96} we say that an orientation preserving homeomorphism $f$ of $\T^n$ has \emph{Denjoy-type} if there exists a continuous map $h$ of $\T^n$ to itself, such that 

\no
-- $h$ is homotopic to the identity,  

\no
-- the set $V_h$ of \emph{non-trivial values} of $h$ (i.e., elements $p\in\T^n$ such that $\#\{h^{-1}(p)\}>1$) is countable,

\no
-- there is a minimal translation $R$ of $\T^n$  such that $f$ and $R$ are semi-conjugate via the map $h$, i.e.,
$$
h\circ f=R\circ h,
$$  

\no 
-- the map $f$ is not conjugate to a minimal translation.

Let us suppose that $f$ is a Denjoy-type homeomorphism of $\T^n$, let $h$ be a semi-conjugation map as above, and let
$$
\Lambda=\T^n\setminus\cup_{p\in V_h}\{{\rm int}(h^{-1}(p))\},
$$
where ${\rm int}(A)$ stands for the interior of a set $A$. 
From this definition we conclude that $\Lambda$ is a closed subset of $\T^n$ that is \emph{completely  
invariant}   (i.e., forward and backward invariant),
and \emph{minimal} (i.e., the orbit of each point of the set is dense in this set) under the map $f$. Also, if there exists $p\in V_h$ such that ${\rm int}(h^{-1}(p))\neq\emptyset$, then $\Lambda$ is nowhere dense in $\T^n$. Finally, each $\Omega_p={\rm int}(h^{-1}(p))$ is a \emph{wandering domain} for $f$, i.e., $f^m(\Omega_p)=\Omega_p,\ m\in\Z$, implies $m=0$. 

We say that $f$ has \emph{round} wandering domains if
$$
\R^n\setminus\tilde\Lambda=\cup_{i=1}^\infty  B_i,
$$
where $\tilde\Lambda=\psi^{-1}(\Lambda)$, for each $i$, the set $B_i$ is an open Euclidean ball with non-empty interior, and the balls $B_i$ are pairwise disjoint, i.e., $B_i\cap B_j=\emptyset,\ i\neq j$.  

\subsection{Quasiconformal and quasisymmetric maps}\label{SS:QcQs}
In the proof of Theorem~\ref{T:Main} we use basic properties of quasiconformal and quasisymmetric maps. Let $\Sph^n$ denote the \emph{standard $n$-sphere}, i.e., the unit sphere in $\R^{n+1}$  endowed with the chordal metric. In what follows, we assume that $n\geq 2$.
 If $f$ is a homeomorphism between open regions $U$ and $V$ in $\R^n$ or  $\Sph^n$, its \emph{dilatation at a point} $p$ is defined as 
$$
K_f(p)=\limsup_{r\to0}\frac{L(p,r)}{l(p,r)},
$$   
where
\begin{equation}\notag
\begin{split}
L(p,r)&=\sup\{|f(q)-f(p)|\: |q-p|=r\},\\ 
l(p,r)&=\inf\{|f(q)-f(p)|\: |q-p|=r\}.
\end{split}
\end{equation}
We say that $f$ is \emph{quasiconformal} if for its \emph{dilatation} $K_f=\sup_{p\in U}K_f(p)$ we have $K_f\leq K<\infty$.  In this case we also say that $f$ is $K$-\emph{quasiconformal}. If $K_f=1$, the map $f$ is called \emph{conformal}.  

Since $n$-torus $\T^n$ is compact, if $f$ is a $C^1$-diffeomorphism of $\T^n$, then  its lift $\tilde f$ to $\R^n$ under $\psi$ is a $K$-quasiconformal homeomorphism for some $K$. Indeed, this follows because $\psi$ is a local isometry and the dilatation $K_{\tilde f}(p)$ depends only on local properties of $\tilde f$, so $K_{\tilde f}$ is a continuous doubly periodic  function and therefore bounded.  

A homeomorphism $f$ between metric spaces $(X, d_X)$ and $(Y, d_Y)$ is called \emph{quasisymmetric} if there exists a homeomorphism $\eta$ of $[0,\infty)$ onto itself, such that
$$
\frac{d_Y(f(x),f(x'))}{d_Y(f(x),f(x''))}\leq\eta\left(\frac{d_X(x,x')}{d_X(x',x'')}\right)
$$
for all distinct triples $x,x'$, and $x''$ in $X$. In this case $\eta$ is called a \emph{distortion function} of $f$ and $f$ is said to be $\eta$-\emph{quasisymmetric}.
It is well known, see, e.g., \cite{He01}, that if $f$ is a homeomorphism of $\R^n$ or $\Sph^n$, then it is quasiconformal if and only if it is quasisymmetric.

An invertible linear map $L$ of $\R^n$  is always quasiconformal. Its {dilatation} $K_L$ is constant and it is given by  
$$
K_L=\frac{\max_{|x|=1}\{|L(x)|\}}{\min_{|y|=1}\{|L(y)|\}}.
$$
Such a map $L$ is conformal if $K_L=1$. 
The conformality of $L$  is equivalent to the statement that $L$ takes a non-degenerate geometric ball to another such ball. In this case $L$ can be written as $L=\lambda T$, where $\lambda>0$ and $T$ is a linear isometry. Here we allow for conformal maps to reverse the orientation. 

\subsection{Schottky sets and groups}\label{SS:SSG}
\no
In this section we review several notions and results from~\cite{BKM09} that will be used in this note. We refer the reader to that paper for more details.

By a \emph{Schottky set} $\Lambda$ in the Euclidean $n$-space $\R^n$ or the standard sphere $\Sph^n$ we mean the complement of a union of pairwise disjoint open geometric balls $B_i,\ i\in I$, in $\R^n$, respectively $\Sph^n$. We often refer to the balls $B_i,\ i\in I$, as \emph{removed} and assume that there are at least three of them. A \emph{peripheral sphere} of a Schottky set $\Lambda$ in $\R^n$ or $\Sph^n$ is the boundary sphere $\dee B_i$ of a removed ball $B_i,\ i\in I$. 

With each Schottky set $\Lambda$ in $\Sph^n$ one can associate a \emph{Schottky group} $\Gamma_\Lambda$. This group is generated by reflections $\gamma_i$ in corresponding peripheral spheres $\dee B_i,\ i\in I$. Each such group $\Gamma_\Lambda$ is a subgroup of M\"obius transformations with generators $\gamma_i,\ i\in I$, and relations $\gamma_i^2=\id,\ i\in I$.

In the proof of the main result below we use the following two lemmas from~\cite{BKM09}.

\begin{lemma}\cite[Proposition~5.5]{BKM09}\label{L:EqvExt}
Let $f\:\Lambda\to\Lambda'$ be a quasisymmetric map between two Schottky sets in $\Sph^n$. 
Then $f$ has an equivariant quasiconformal extension $f_\Lambda$ to $\Sph^n$ with respect to the groups $\Gamma_{\Lambda}$ and $\Gamma_{\Lambda'}$. I.e., for each $\gamma\in\Gamma_\Lambda$ there exists $\gamma'\in\Gamma_{\Lambda'}$ such that $f_\Lambda\circ \gamma=\gamma'\circ f_\Lambda$.
\end{lemma}

\begin{lemma}\cite[Lemma~6.1]{BKM09}\label{L:NestedBalls}
Suppose that $f$ is a continuous map of $\R^n,\ n\in\N$, to itself that is differentiable at the origin $0$. Suppose further that there is a sequence of open geometric balls $(B_i)_{i\in\N}$ that contain 0, such that ${\rm diameter}(B_i)\to0,\ i\to\infty$, and for each $i\in\N$, the set  $f(B_i)$ is a geometric ball. Then the derivative $D_0f$ is a (possibly degenerate or orientation reversing)  conformal linear map, i.e., $D_0f=\lambda T$, where $\lambda\geq 0$ and $T$ is a linear isometry.  
\end{lemma}

For reader's convenience, we outline the proofs of these lemmas and refer to~\cite{BKM09} for more details. 

To prove Lemma~\ref{L:EqvExt}, we let $\dee B_i,\ i\in I$, be  a peripheral sphere for $\Lambda$ with the largest radius, and let $\dee B_i'$ be the peripheral sphere for $\Lambda'$ that corresponds to $\dee B_i$ under the given quasisymmetry $f$. Note here that peripheral spheres of a Schottky set are characterized topologically by the property that their removal does not separate the given Schottky set. 

Let $\gamma_i\in\Gamma_\Lambda$ be the reflection in $\dee B_i,\ i\in I$, and $\gamma_i'\in\Gamma_{\Lambda'}$ be the corresponding reflection in $\dee B_i'$. We use $\gamma_i$ and $\gamma_i'$ to double the Schottky sets $\Lambda$ and $\Lambda'$ across $\dee B_i$ and $\dee B_i'$, respectively, and denote the doubled spaces by $\Lambda_i$ and $\Lambda_i'$, respectively. Both, $\Lambda_i$ and $\Lambda_i'$ are Schottky sets since $\gamma_i$ and $\gamma_i'$ are M\"obius transformations. We next extend the map $f$ to $\Lambda_i\setminus\Lambda$ by the formula 
\begin{equation}\label{E:Ext}
f=\gamma_i'\circ f\circ\gamma_i^{-1}. 
\end{equation}
We now replace the Schottky set $\Lambda$ by $\Lambda_i$, the Scottky set $\Lambda'$ by $\Lambda_i'$,  and continue this doubling process  indefinitely. 

It is not hard to see that in this way we obtain a new map $f_\Lambda$ defined on a dense subset of $\Sph^n$ and that is equivariant with respect to $\Gamma_\Lambda$, i.e., 
$$
f_\Lambda\circ \gamma=\gamma'\circ f_\Lambda
$$ 
for each $\gamma\in\Gamma_\Lambda$ and the corresponding $\gamma'\in\Gamma_{\Lambda'}$. The last property follows from~\eqref{E:Ext}.  It is also straightforward that such a map extends to a homeomorphism of $\Sph^n$. Indeed, the doubling process across the peripheral spheres described above guarantees that, for any given $r>0$, all the radii of peripheral spheres in the domain and target Schottky sets would eventually (after doubling a certain number of times) be less than $r$. 

The hard part is to show that the resulting map $f_\Lambda$ is quasiconformal. We do this by  first showing that the given quasisymmetry $f\:\Lambda\to\Lambda'$ has a $K$-quasiconformal extension to all of $\Sph^n$. This can be done by applying the classical Ahlfors--Beurling extension result for each pair $B_i, B_i'$ of removed balls, such that $f\:\dee B_i\to\dee B_i'$. Since $B_i$ and $B_i',\ i\in I$, are geometric balls, it is not hard to see that the resulting global extension is $K$-quasiconformal for some $K$.  We next show that using the formula~\eqref{E:Ext} to redefine a given quasiconformal map across $\dee B_i$ does not change the dilatation $K$. This is elementary as $\gamma_i$ and $\gamma_i'$ are M\"obius transformations and therefore have dilatations equal to 1. Finally, we use standard compactness arguments for families of normalized uniformly quasiconformal maps.

The proof of Lemma~\ref{L:NestedBalls} is elementary. Indeed, let $r_i$ be the radius of the ball $B_i,\ i\in\N$, and consider the sequence of rescaled balls $(B_i/r_i)_{i\in\N}$. Each such rescaled ball has radius 1, and, by possibly passing to a subsequence, we may assume that the sequence $(\overline{B_i}/r_i)_{i\in\N}$ of closed balls converges in the Hausdorff sense. The limit must be a closed geometric ball $\overline B$ of radius 1 that contains the origin. Now, the sequence of rescaled maps $(f_i)_{i\in\N}$, where 
$$
f_i(x)=f(r_ix)/r_i,
$$ 
converges locally uniformly to the linear map $D_0f$. This follows from the assumptions that $f$ is differentiable at 0, and $r_i\to0$ as $i\to\infty$. Furthermore, our assumption gives that
$$
f_i(\overline{B_i}/r_i)=f(\overline{B_i})/r_i
$$
is a closed geometric ball $\overline{B_i'}$. Thus  the  Hausdorff limit of $(\overline{B_i'})_{i\in\N}$ is also a closed  geometric ball $\overline{B'}$, possibly degenerate. This limit must be the image of $\overline B$ under $D_0f$. If it is not degenerate, we conclude that $D_0f$ is conformal.

\section{Proof of Theorem~\ref{T:Main}}\label{S:Proof}
\no
We argue by contradiction and assume that there exists a Denjoy-type $C^1$-dif\-feo\-mor\-phism $f$ of $\T^n$  whose wandering domains are round.
The proof of Theorem~\ref{T:Main} is a simple consequence of the following two lemmas.
\begin{lemma}\label{L:ConfMin}
A lift $\tilde f$ of $f$ under $\psi$ is conformal at each point of  $\tilde\Lambda=\psi^{-1}(\Lambda)\subseteq\R^n$. I.e., for each $p\in\tilde\Lambda$ we have $D_p\tilde f=\lambda T$, where $\lambda>0$ and $T$ is a linear isometry. 
\end{lemma}
\begin{proof}
This is equivalent to proving that for each $p\in\tilde\Lambda$, the dilatation of $D_p\tilde f$ equals 1. 

Since $\Lambda$ is nowhere dense,  so is $\tilde\Lambda$. Thus for each $p\in\tilde\Lambda$ there exists a sequence of complementary balls $(B_i)_{i\in\N}$ of $\tilde\Lambda$ that accumulate at $p$, i.e., 
$$
{\rm diameter} B_i\to0\quad {\rm and}\quad d_{\rm Hausd}(\overline{B_i},\{p\})\to0\  {\rm as}\  i\to\infty,
$$  
where $d_{\rm Hausd}$ denotes the Hausdorff distance.
Also, since $f$ is $C^1$-dif\-fe\-ren\-ti\-able on a compact set $\T^n$,  we have
\begin{equation}\label{E:Diff}
\tilde f(q+x)-\tilde f(q)-D_q\tilde f(x)=o(x),\quad x\to0,
\end{equation}
where $|o(x)|/|x|\to0$ as $|x|\to0$, uniformly in $q$. This claim follows immediately from the 
integral form of the Mean Value Theorem for vector-valued functions. 

Now, in~\eqref{E:Diff} we choose $q=q_i$ to be the center of the ball $B_i$, and $|x|=r_i$, where $r_i$ is the radius of $B_i,\ i\in\N$. Then our assumption that the wandering domains are round implies that the image of the peripheral  sphere $\dee B_i$ under the map 
$$
x\mapsto \tilde f(q_i+x)-\tilde f(q_i)
$$ 
is a sphere $S_i$ that encloses 0. Therefore, the limit of the rescaled spheres $S_i/r_i,\ i\in\N$, is a sphere that encloses 0, but that is possibly degenerate, i.e., its radius may be 0 or $\infty$. However, according to~\eqref{E:Diff}, this limit must be the image of the unit sphere under the linear map $D_p\tilde f$, which is non-degenerate because $f$ is a diffeomorphism. This implies that the limit sphere is non-degenerate and the map $D_p\tilde f$ is conformal.
\end{proof}

In what follows, we identify $\R^n\cup\{\infty\}$ with $\Sph^n$ via stereographic projection. The set $\tilde\Lambda\cup\{\infty\}$ is a Schottky set in $\Sph^n$, and we continue to denote it by $\tilde\Lambda$ to simplify notations. As pointed out in Subsection~\ref{SS:QcQs}, the map $\tilde f$ is $K$-quasiconformal in $\R^n$ for some $K\geq1$. 
Therefore it is also a quasiconformal map of $\Sph^n$ onto itself that fixes $\infty$. For homeomorphisms of $\Sph^n$ being quasiconformal and quasisymmetric are equivalent. We conclude that $\tilde f$ is a quasisymmetric map of $\Sph^n$, and hence so is its restriction to the Schottky set $\tilde\Lambda$. From Lemma~\ref{L:EqvExt} we now have that $\tilde f$ can be redefined on $\Sph^n\setminus\tilde\Lambda$ such that it becomes an equivariant quasiconformal map, denoted $\tilde f_{\tilde\Lambda}$.

\begin{lemma}\label{L:QCConf}
The map $\tilde f_{\tilde\Lambda}$ is conformal everywhere.
\end{lemma}
\begin{proof}
Since the map $\tilde f_{\tilde\Lambda}$ is quasiconformal, it is differentiable almost everywhere. We prove that $\tilde f_{\tilde\Lambda}$ is conformal at almost every point $p$ of differentiability. 

Since we do not assume whether $\Lambda$ has positive or zero measure, we need to consider the following two cases: 

\noindent
-- $p\in \gamma(\tilde\Lambda)$ for some $\gamma\in\Gamma_{\tilde\Lambda}$, or 

\noindent
-- there exists a sequence $(\gamma_k)_{k\in\N}$ of elements of $\Gamma_{\tilde\Lambda}$ such that $p\in B_{i_k},\ k\in\N$, where $B_{i_k}$  is a removed ball for $\gamma_k(\tilde\Lambda)$. 

We start with the first case, i.e.,  $p\in \gamma(\tilde\Lambda)$ for some $\gamma\in\Gamma_{\tilde\Lambda}$.
In this case we may assume that $p$ is a Lebesgue density point of $\gamma(\tilde\Lambda)$ because such points form a set of full measure. Furthermore, since $\gamma$ is conformal and $\tilde f_{\tilde \Lambda}$ is $\Gamma_{\tilde\Lambda}$-equivariant, it is enough to prove that $\tilde f_{\tilde \Lambda}$ is conformal at each Lebesgue density point $p$ of $\tilde\Lambda$.  We have
$$
\lim_{r\to0}\frac{|\tilde\Lambda\cap B(p,r)|}{|B(p,r)|}=1,
$$
where $|A|$ stands for the Lebesgue measure of a set $A$ in an appropriate dimension.
An elementary application of the Coarea formula gives 
\begin{equation}\label{E:Lebesgue}
\lim_{k\to\infty}\frac{|\tilde\Lambda\cap \dee B(p,r_k)|}{|\dee B (p,r_k)|}=1,
\end{equation}
where  $(r_k)_{k\in\N}$ is a sequence of positive numbers tending to 0. 
Since $\tilde f_{\tilde\Lambda}$ is differentiable at $p$, we have
$$
\frac{\tilde f_{\tilde \Lambda}(p+r_kx)-\tilde f_{\tilde\Lambda}(p)}{r_k}=D_p\tilde f_{\tilde\Lambda}(x)+o(1),\quad k\to\infty,
$$
where $|x|=1$. We denote by $f_k$ the map on the left-hand side of this equation. Namely,
$$
\tilde f_k(x)=\frac{\tilde f_{\tilde\Lambda}(p+r_kx)-\tilde f_{\tilde\Lambda}(p)}{r_k}.
$$ 
We thus have
\begin{equation}\label{E:D}
\tilde f_k\to D_p\tilde f_{\tilde\Lambda},\quad k\to\infty,
\end{equation}
uniformly on the unit sphere $\dee B(0,1)$.

Let $x,\ |x|=1$, be arbitrary. From~\eqref{E:Lebesgue} we know that for each $k\in\N$ there exists $x_k\in\tilde\Lambda_k\cap \dee B(0,1)$, where $\tilde\Lambda_k={(\tilde\Lambda-p)}/{r_k}$, such that the sequence $(x_k)_{k\in\N}$ converges to $x$. Convergence in~\eqref{E:D} then gives that $\tilde f_k(x_k)\to D_p\tilde f_{\tilde\Lambda}(x)$. On the other hand, on the set $\tilde\Lambda$, the map $\tilde f_{\tilde\Lambda}$ agrees with $\tilde f$. Therefore, since $\tilde f$ is also differentiable at $p$, we have $\tilde f_k(x_k)\to D_p\tilde f(x)$. We hence arrive at the equality
$$
D_p\tilde f_{\tilde\Lambda}=D_p\tilde f.
$$
Since $\tilde f$ is conformal at $p$, we conclude that $\tilde f_{\tilde\Lambda}$ is also conformal at $p$.

We now deal with the case $\{p\}=\cap_{k\in\N} B_{i_k},\ {\rm diameter}B_{i_k}\to0$ as $k\to\infty$, where $B_{i_k},\ k\in\N$,  is a removed ball for $\gamma_k(\tilde\Lambda)$ with $\gamma_k\in\Gamma_{\tilde\Lambda}$. Without loss of generality we may assume that $p=\tilde f_{\tilde\Lambda}(p)=0$. Because $\tilde f_{\tilde\Lambda}$ is equivariant with respect to $\Gamma_{\tilde\Lambda}$, for all $k\in\N$, it takes removed balls of $\gamma_k(\tilde\Lambda)$  to removed balls of $\gamma_k'(\tilde\Lambda)$, for some $\gamma_k'\in\Gamma_{\tilde\Lambda}$. Now, since all the removed balls are geometric balls, the conformality of $\tilde f_{\tilde\Lambda}$ at $p$ follows from Lemma~\ref{L:NestedBalls}.

 Since $\tilde f_{\tilde\Lambda}$ is quasiconformal in $\R^n$ and conformal at almost every point, it is conformal everywhere. 
\end{proof}

The proof of Theorem~\ref{T:Main} now follows immediately from Lemmas~\ref{L:ConfMin} and~\ref{L:QCConf}.
Indeed, since $\tilde f_{\tilde\Lambda}$ is conformal and $\tilde f_{\tilde\Lambda}(\infty)=\infty$, it has to be of the form $\tilde f_{\tilde\Lambda}(x)=\lambda Tx+a$ in  $\R^n$, where $\lambda>0$ and $T$ is a linear isometry. In addition, since $f$ is a homeomorphism of $\T^n$, the map $\tilde f$ is its lift to $\R^n$,  and $\tilde f_{\tilde\Lambda}|_{\tilde\Lambda}=\tilde f|_{\tilde\Lambda}$, we must have that $\tilde f_{\tilde\Lambda}$ is an isometry, i.e., $\lambda=1$. This implies that $\tilde f_{\tilde\Lambda}$, and hence $\tilde f$ cannot change the sizes (i.e., the radii)  of removed balls. This is clearly a contradiction to the assumption that these balls project under $\psi$ to wandering domains of $f$.
\qed

\section{Concluding remarks}
\no
As mentioned in the introduction, for the maps constructed by  P.~D.~McSwiggen in~\cite{McS93, McS95} the wandering domains have non-uniform geometry. 
We finish this note by stating the following conjecture.
\begin{conjecture}\label{C:QC}
If $n\geq 2$, there is no Denjoy-type quasiconformal homeomorphism $f$ of $\T^n$ all of whose wandering domains are quasi-round.
\end{conjecture}
For wandering domains to be \emph{quasi-round} it means that the complement of the minimal set $\Lambda$ is the disjoint union of  \emph{uniform quasiballs}, i.e., images of geometric balls under global $K$-quasiconformal maps with the same dilatation $K$.

The proof of Theorem~\ref{T:Main} above under the assumption that $f$ is only quasiconformal breaks down even if we assume that all the wandering domains are round. Indeed, without the $C^1$-differentiability assumption on $f$ we  cannot argue that a lift $\tilde f$ of $f$ to $\R^n$  is conformal at each point of $\tilde\Lambda=\psi^{-1}(\Lambda)$. However, if $\Lambda$ happens to have measure zero, then the above arguments apply and Conjecture~\ref{C:QC} holds in this case (i.e., assuming round wandering domains and $|\Lambda|=0$). This essentially follows from~\cite[Theorem~1.1]{BKM09}.

Also, let us assume that $n=2$ and the wandering domains are uniform quasiballs (or quasidiscs in this case) that are pairwise relatively separated (i.e., pairwise relative distances are bounded away from 0). 	Then one can use arguments as in~\cite{Bo11} to show that there is a quasiconformal map $\phi$ of $\T^2$ to another torus $\R^2/\mathcal L$, where $\mathcal L$ is a lattice in $\R^2$, such that the wandering domains of the conjugate  quasiconformal map 
$$
f_\phi=\phi\circ f\circ\phi^{-1}
$$ 
are round. If, in addition, we assume that $|\Lambda|=0$, then for the minimal set $\Lambda_\phi=\phi(\Lambda)$ of $f_\phi$ we have $|\Lambda_\phi|=0$, because quasiconformal maps send sets of measure zero to sets of measure zero. We thus conclude that Conjecture~\ref{C:QC} holds in this case as well, i.e., when $n=2$ and the wandering domains are uniformly separated uniform quasidiscs.

The above suggests that the conjecture should be true at least in the case $|\Lambda|=0$. 
In the opposite direction, let us assume that $n=2$ and $\Lambda\subseteq\T^2$ is such that $|\Lambda|>0$, and all the complementary components of its lift $\tilde\Lambda=\psi^{-1}(\Lambda)$ are  geometric discs. By choosing a non-trivial doubly periodic Beltrami coefficient $\mu$ on $\tilde\Lambda$, say $\mu\equiv1/2$, we can find a non-trivial  quasiconformal deformation $\tilde f$ of $\tilde\Lambda$ onto $\tilde\Lambda'$ so that all the complementary components of $\tilde\Lambda'$ are geometric discs. 

To achieve this, one needs to extend $\mu$ to all of $\Sph^2$ equivariantly with respect to the Schottky group $\Gamma_{\tilde\Lambda}$, and then solve the Beltrami equation. This construction gives an arbitrary quasiconformal deformation between two sets in 2-tori with round complementary components. To be able to produce a counterexample to Conjecture~\ref{C:QC} in this manner, one needs to ensure that
$\tilde\Lambda'=\tilde\Lambda$ and the map $\tilde f$ descends under $\psi$ to a Denjoy-type homeomorphism $f$ of $\T^2$. Either of these conditions appears to be non-trivial to verify.

\end{document}